\documentclass[11pt]{article}

\usepackage{amsmath,amssymb,amsthm}
\usepackage[a4paper,margin=1in]{geometry}
\usepackage[hidelinks]{hyperref}

\newtheorem{theorem}{Theorem}[section]
\newtheorem{proposition}[theorem]{Proposition}
\newtheorem{lemma}[theorem]{Lemma}
\newtheorem{corollary}[theorem]{Corollary}
\theoremstyle{definition}

\theoremstyle{remark}
\newtheorem{remark}[theorem]{Remark}

\DeclareMathOperator{\tr}{tr}

\DeclareMathOperator{\dist}{dist}
\newcommand{\HS}{\mathrm{HS}}
\newcommand{\T}{\mathsf{T}}
\newcommand{\R}{\mathbb{R}}
\newcommand{\bu}{\mathbf u}
\newcommand{\be}{\mathbf e}
\newcommand{\bv}{\mathbf v}
\newcommand{\bw}{\mathbf w}
\newcommand{\bz}{\mathbf z}
\newcommand{\bq}{\mathbf q}
\newcommand{\bff}{\mathbf f}

\newcommand{\norm}[1]{\lVert #1 \rVert}

\title{A Projection Identity for Simplices\\
Sharp Inequalities, Converse Results, and Affine Projections}
\author{Quang Hung Tran}
\date{}

\begin{document}

\maketitle

\begin{abstract}
We study a projection identity for a simplex in Euclidean space, written in terms of the frame operator of its unit edge directions. For a right simplex, the identity leads to a sharp family of distance inequalities and a complete description of equality. For a general simplex, the same formula is controlled by the spectrum of the Gram matrix through the Ky Fan principle. We prove converse results that characterise right simplices and determine the smallest number of projection subspaces needed to force orthogonality, together with an optimal quantitative estimate. We also treat affine projection subspaces and show how the original inequality for mutually perpendicular vectors fits into the same framework.
\end{abstract}



\section{Introduction}\label{sec:intro}

Throughout, $SA_1\ldots A_n$ denotes a simplex in Euclidean $n$-space $\R^n$, and we
identify $S$ with the origin. We call the simplex \emph{right at $S$} if the edges issuing
from $S$ are mutually perpendicular, so
\[
SA_i\perp SA_j \text{ whenever } i\ne j.
\]
For a right simplex, the $n$ orthogonal edges at $S$ determine an orthotope. If $P$ is the
vertex opposite $S$, then
\[
SP^2=\sum_{i=1}^{n}SA_i^2.
\]
For a general simplex we use the abbreviation
\[
\rho^2:=\sum_{i=1}^{n}SA_i^2,
\]
so that $\rho=SP$ in the right case.

Right simplices provide a natural higher-dimensional setting for Pythagorean identities.
The classical $n$-dimensional Pythagorean theorem and several of its proofs may be found in
\cite{DonchianCoxeter,ConantBeyer,LinLin,Porter,EiflerRhee,Cook,Drucker,TranPythagoras}.
In three dimensions, de Gua's theorem is the corresponding statement for a tetrahedron with
three mutually perpendicular edges at one vertex; a related vector generalisation was
considered in \cite{TranDeGua}. For general background on Euclidean and simplex geometry,
see \cite{Berger,Coxeter,Fiedler}. Related inequalities and projection constructions for
simplices appear in \cite{TranKlamkin,TranVanAubel,TranGrace}.

The present work began with the following elementary inequality. If $SA_1\ldots A_n$ is
right at $S$, $L$ is a line through $S$, and $H_i$ is the orthogonal projection of $A_i$
onto $L$, then
\begin{equation}\label{eq:motivating}
\sum_{i=1}^{n}A_iH_i\le\sqrt{n-1}\,SP.
\end{equation}
The question that led to this paper is what lies behind \eqref{eq:motivating}. The
basic calculation is elementary. For an orthonormal basis $\be_1,\ldots,\be_n$ and a linear map
$T$,
\[
\sum_{i=1}^{n}\norm{T\be_i}^2=\norm{T}_{\HS}^2.
\]
For orthogonal projections this combines with a trace identity. The interest lies in the
geometric consequences of this simple observation.

We first obtain a master identity valid for arbitrary nonzero edge vectors. For right
simplices it gives a sharp one-parameter family of distance inequalities. For general
simplices, the same identity is governed by the extreme partial sums of the Gram
spectrum through the Ky Fan principle. We then prove two converse results. The first
characterises right simplices, while the second shows that exactly
\[
\frac{n(n+1)}2-1
\]
subspaces of any fixed dimension $1\le k\le n-1$ are necessary and sufficient, in the
optimal sense, to detect orthogonality. We also obtain the best constants in a quantitative
estimate for the departure from orthogonality. Finally, we treat affine projection
subspaces and return to the vector inequality that motivated the paper.

\section{Notation and the master identity}\label{sec:master}

\subsection{Standing notation}

Let $\bu_1,\ldots,\bu_n$ be nonzero vectors in $\R^n$, and write
\[
a_i=\norm{\bu_i}, 
\be_i=\frac{\bu_i}{a_i}.
\]
When the vectors are linearly independent, we identify them with the edge vectors
$\overrightarrow{SA_i}$ of a simplex $SA_1\ldots A_n$. Put
\[
E=[\,\be_1\ \cdots\ \be_n\,], 
M=EE^{\T}=\sum_{i=1}^{n}\be_i\be_i^{\T}, 
G=E^{\T}E=\bigl(\langle\be_i,\be_j\rangle\bigr)_{i,j=1}^{n}.
\]
Thus $M$ is the frame operator and $G$ is the Gram matrix of the unit vectors. Both are
positive semidefinite, they have the same spectrum, and
\begin{equation}\label{eq:trM}
\tr M=\tr G=n.
\end{equation}
If the vectors are linearly independent, both matrices are positive definite. Moreover,
$M=I$ if and only if $E$ is orthogonal, equivalently if and only if the vectors $\bu_i$ are
mutually perpendicular.

Let $L\subseteq\R^n$ be a $k$-dimensional subspace and let $Q$ be the orthogonal projection
onto $L$. If $A_i=S+\bu_i$, let $H_i$ be the orthogonal projection of $A_i$ onto $L$, and
for a real parameter $\lambda$ define $B_i$ by
\[
\overrightarrow{SB_i}=\lambda\bu_i.
\]
No linear independence is needed for the identities in this section.

\subsection{The operator identity}

\begin{theorem}[Operator form]\label{thm:operator}
Let $V$ be an $n$-dimensional Euclidean vector space and let
$\bu_1,\ldots,\bu_n$ be a nonzero orthogonal basis of $V$. If $T\colon V\to W$ is linear,
where $W$ is a finite-dimensional Euclidean vector space, then
\begin{equation}\label{eq:HS}
\sum_{i=1}^{n}\frac{\norm{T\bu_i}^{2}}{\norm{\bu_i}^{2}}=\norm{T}_{\HS}^{2},
\end{equation}
where $\norm{T}_{\HS}^{2}=\tr(T^{*}T)$. Consequently,
\begin{equation}\label{eq:HSCS}
\sum_{i=1}^{n}\norm{T\bu_i}
\le
\norm{T}_{\HS}\Bigl(\sum_{i=1}^{n}\norm{\bu_i}^{2}\Bigr)^{1/2},
\end{equation}
with equality if and only if there is a constant $c\ge0$ such that
\[
\norm{T\bu_i}=c\,\norm{\bu_i}^2 \text{ for } 1\le i\le n.
\]
\end{theorem}

\begin{proof}
The vectors $\be_i=\bu_i/\norm{\bu_i}$ form an orthonormal basis of $V$, hence
\[
\sum_{i=1}^{n}\frac{\norm{T\bu_i}^{2}}{\norm{\bu_i}^{2}}
=
\sum_{i=1}^{n}\norm{T\be_i}^{2}
=
\norm{T}_{\HS}^{2}
\]
by the basis invariance of the Hilbert--Schmidt norm \cite[Ch.~5]{HornJohnson}. Applying the
Cauchy--Schwarz inequality \cite{HLP} to
$\sum_i(\norm{T\bu_i}/\norm{\bu_i})\norm{\bu_i}$ gives \eqref{eq:HSCS} and its equality
case.
\end{proof}

\begin{lemma}\label{lem:trace}
For the unit vectors $\be_1,\ldots,\be_n$ defined above,
\[
\sum_{i=1}^{n}\norm{Q\be_i}^{2}=\tr(QM).
\]
\end{lemma}

\begin{proof}
Since $Q=Q^{*}=Q^2$,
\[
\norm{Q\be_i}^{2}
=
\langle Q\be_i,\be_i\rangle
=
\tr\bigl(Q\be_i\be_i^{\T}\bigr).
\]
Summing over $i$ gives the result.
\end{proof}

\begin{proposition}[Master identity]\label{prop:master}
For arbitrary nonzero vectors $\bu_1,\ldots,\bu_n$ in $\R^n$, every subspace $L$ through
$S$, and every real $\lambda$,
\begin{equation}\label{eq:master}
\sum_{i=1}^{n}\frac{B_iH_i^{2}}{SA_i^{2}}
=
n\lambda^{2}+(1-2\lambda)\tr(QM).
\end{equation}
\end{proposition}

\begin{proof}
Because $\overrightarrow{SH_i}=Q\bu_i$ and
$\overrightarrow{SB_i}=\lambda\bu_i$,
\[
\frac{B_iH_i^{2}}{SA_i^{2}}
=
\norm{(Q-\lambda I)\be_i}^{2}
=
\lambda^{2}+(1-2\lambda)\norm{Q\be_i}^{2}.
\]
Summing and applying Lemma~\ref{lem:trace} proves \eqref{eq:master}.
\end{proof}

\begin{remark}[Ambient dimension]\label{rem:ambient}
The assumption that the $n$ edge vectors span the ambient $n$-space matters only when one
replaces $\tr(QM)$ by the dimension $k$. If mutually perpendicular edges span an
$n$-dimensional subspace $V$ of a larger space $\R^N$, then $M=P_V$ and
\[
\tr(QM)=\tr(QP_V)
=
\sum_{j=1}^{\min(k,n)}\cos^2\vartheta_j,
\]
where $\vartheta_j$ are the principal angles between $L$ and $V$ \cite{BjorckGolub}. In
particular, $\tr(QM)=k$ precisely when $L\subseteq V$.
\end{remark}

\section{Right simplices and the sharp inequality}\label{sec:right}

For a right simplex, $M=I$ and $\tr(QM)=k$. The master identity therefore becomes
independent of the position of the projection subspace.

\begin{theorem}\label{thm:right}
Let $SA_1\ldots A_n$ be right at $S$ in $\R^n$, let $L$ be a $k$-dimensional subspace
through $S$ with $1\le k\le n-1$, and let $\lambda$ be real. Then
\begin{equation}\label{eq:identity}
\sum_{i=1}^{n}\frac{B_iH_i^{2}}{SA_i^{2}}
=
n\lambda^{2}-2k\lambda+k,
\end{equation}
and
\begin{equation}\label{eq:inequality}
\sum_{i=1}^{n}B_iH_i
\le
\sqrt{n\lambda^{2}-2k\lambda+k}\,SP.
\end{equation}
Equality in \eqref{eq:inequality} holds if and only if there is a constant $c>0$ such that
\begin{equation}\label{eq:equalitycase}
SA_i=c\,\norm{(Q-\lambda I)\be_i} \text{ for } 1\le i\le n.
\end{equation}
\end{theorem}

\begin{proof}
Equation \eqref{eq:identity} follows from Proposition~\ref{prop:master}. Applying
Theorem~\ref{thm:operator} to $T=Q-\lambda I$ gives
\[
\norm{T}_{\HS}^{2}
=
\tr\bigl((Q-\lambda I)^2\bigr)
=
k(1-\lambda)^2+(n-k)\lambda^2
=
n\lambda^2-2k\lambda+k.
\]
Since $\sum_iSA_i^2=SP^2$, \eqref{eq:inequality} and \eqref{eq:equalitycase} follow.
\end{proof}

\begin{corollary}\label{cor:sharp}
The constant in \eqref{eq:inequality} is best possible for every real $\lambda$. More
precisely, fix an orthonormal system of edge directions and a $k$-dimensional subspace $L$.
If $\lambda$ is different from both $0$ and $1$, then for every position of $L$ one can choose positive edge
lengths so that equality holds. If $\lambda=0$, equality is attainable for a given $L$
if and only if no edge direction is orthogonal to $L$; if $\lambda=1$, it is attainable if
and only if no edge direction lies in $L$.
\end{corollary}

\begin{proof}
If $(Q-\lambda I)\be_i=0$, then $\lambda$ is an eigenvalue of the orthogonal projection
$Q$, hence $\lambda=0$ or $\lambda=1$. Thus, when $\lambda$ is different from both $0$ and $1$, all quantities on the
right of \eqref{eq:equalitycase} are positive and may be used as edge lengths.

For $\lambda=0$, the quantity $\norm{Q\be_i}$ vanishes exactly when $\be_i\perp L$; for
$\lambda=1$, the quantity $\norm{(I-Q)\be_i}$ vanishes exactly when $\be_i$ lies in $L$. The
stated conditions are therefore necessary and sufficient by \eqref{eq:equalitycase}.
To see that the exceptional cases are attainable, put
$\bff=n^{-1/2}(\be_1+\cdots+\be_n)$. For $\lambda=0$, take $L$ to contain $\bff$; then
$Q\be_i\ne0$ for every $i$. For $\lambda=1$, take $L\subseteq\bff^\perp$; then no $\be_i$
lies in $L$. Thus the constant is always sharp.
\end{proof}

\subsection{Special values of \texorpdfstring{$\lambda$}{lambda}}

\begin{corollary}[$\lambda=1$]\label{cor:lambda1}
Let $H_i$ be the orthogonal projection of $A_i$ onto a $k$-dimensional subspace through $S$.
Then
\begin{equation}\label{eq:lambda1}
\begin{aligned}
\sum_{i=1}^{n}\frac{A_iH_i^{2}}{SA_i^{2}}&=n-k,\\
\sum_{i=1}^{n}A_iH_i&\le\sqrt{n-k}\,SP.
\end{aligned}
\end{equation}
and the constant $\sqrt{n-k}$ is best possible. In particular, $k=1$ gives
\eqref{eq:motivating}, while $k=n-1$ gives $\sum_iA_iH_i\le SP$.
\end{corollary}

\begin{corollary}[$\lambda=0$]\label{cor:lambda0}
With the same notation,
\begin{equation}\label{eq:lambda0}
\begin{aligned}
\sum_{i=1}^{n}\frac{SH_i^{2}}{SA_i^{2}}&=k,\\
\sum_{i=1}^{n}SH_i&\le\sqrt{k}\,SP.
\end{aligned}
\end{equation}
and the constant $\sqrt{k}$ is best possible.
\end{corollary}

The two identities are complementary. Indeed,
$A_iH_i^2+SH_i^2=SA_i^2$ for each $i$. For $k=1$, \eqref{eq:lambda0} is the classical
relation $\sum_i\cos^2\alpha_i=1$ for direction cosines \cite[Ch.~18]{Coxeter}.

\begin{corollary}[Optimal parameter]\label{cor:optimal}
Since
\[
n\lambda^{2}-2k\lambda+k
=
n\Bigl(\lambda-\frac{k}{n}\Bigr)^2+\frac{k(n-k)}{n},
\]
the coefficient in \eqref{eq:inequality} is smallest for $\lambda=k/n$. Hence
\[
\sum_{i=1}^{n}B_iH_i
\le
\sqrt{\frac{k(n-k)}{n}}\,SP,
\]
with the constant best possible.
\end{corollary}

\begin{remark}[The midpoint value]\label{rem:midpoint}
For $\lambda=\tfrac12$ the projection term in \eqref{eq:master} disappears and, in fact,
\[
\norm{(Q-\tfrac12 I)\be_i}^2=\frac14
\]
for every $i$, every subspace $L$, and every simplex. If $M_i$ is the midpoint of $SA_i$,
this is simply $M_iH_i=SA_i/2$. The associated inequality reduces to
\[
\sum_{i=1}^{n}SA_i\le\sqrt n\,SP,
\]
so in the linear setting $\lambda=\tfrac12$ carries no information about the projection.
It becomes nontrivial again for affine subspaces in Section~\ref{sec:affine}.
\end{remark}

\section{General simplices and the Gram spectrum}\label{sec:gram}

We now drop the orthogonality hypothesis. Let $\mu_1\ge\mu_2\ge\cdots\ge\mu_n>0$ be the
eigenvalues of $M$, equivalently of the Gram matrix $G$ of the unit edge vectors, and put
\[
\begin{aligned}
\sigma_k^{+}&=\mu_1+\cdots+\mu_k,\\
\sigma_k^{-}&=\mu_{n-k+1}+\cdots+\mu_n.
\end{aligned}
\]
By \eqref{eq:trM} we have $\sigma_n^{+}=\sigma_n^{-}=n$. The Gram matrix is the natural
carrier of metric information for a general simplex; see \cite{Fiedler} for a systematic
treatment.

\begin{theorem}\label{thm:kyfan}
Let $SA_1\ldots A_n$ be a simplex in $\R^{n}$ and let $L$ be a $k$-dimensional subspace
through $S$, $1\le k\le n-1$. Then
\begin{equation}\label{eq:kyfan}
\sigma_k^{-}\le\tr(QM)\le\sigma_k^{+},
\end{equation}
and consequently, if $\lambda<\tfrac12$,
\[
n\lambda^{2}+(1-2\lambda)\sigma_k^{-}
\le
\sum_{i=1}^{n}\frac{B_iH_i^{2}}{SA_i^{2}}
\le
n\lambda^{2}+(1-2\lambda)\sigma_k^{+}.
\]
If $\lambda=\tfrac12$, then
\[
\sum_{i=1}^{n}\frac{B_iH_i^{2}}{SA_i^{2}}=\frac n4.
\]
If $\lambda>\tfrac12$, then
\[
n\lambda^{2}+(1-2\lambda)\sigma_k^{+}
\le
\sum_{i=1}^{n}\frac{B_iH_i^{2}}{SA_i^{2}}
\le
n\lambda^{2}+(1-2\lambda)\sigma_k^{-}.
\]
Equality holds on the right of \eqref{eq:kyfan} if and only if $L$ is invariant under $M$
and the eigenvalues of $M|_L$ are $\mu_1,\ldots,\mu_k$ counted with multiplicity, and
similarly on the left with $\mu_{n-k+1},\ldots,\mu_n$; in particular both bounds are
attained for every simplex.
\end{theorem}

\begin{proof}
The upper bound in \eqref{eq:kyfan} is the Ky Fan maximum principle \cite{Fan}:
\[
\max_{\dim L=k}\tr(P_LM)=\mu_1+\cdots+\mu_k.
\]
See \cite[Ch.~III]{Bhatia} or \cite[Ch.~4]{HornJohnson} for the equality case. Applying the
same principle to $-M$ gives the lower bound. The rest is Proposition~\ref{prop:master}.
\end{proof}

\begin{corollary}\label{cor:gramcases}
For every simplex and every $k$-dimensional $L$ through $S$,
\[
\begin{aligned}
\sigma_k^{-}&\le\sum_{i=1}^{n}\frac{SH_i^{2}}{SA_i^{2}}\le\sigma_k^{+},\\
n-\sigma_k^{+}&\le\sum_{i=1}^{n}\frac{A_iH_i^{2}}{SA_i^{2}}\le n-\sigma_k^{-}.
\end{aligned}
\]
and all four bounds are attained. Moreover
\[
\sum_{i=1}^{n}B_iH_i\le\sqrt{\,n\lambda^{2}+(1-2\lambda)\sigma_k^{\varepsilon}\,}\;\rho,
\]
where $\rho^{2}=\sum_i SA_i^{2}$, and $\varepsilon$ is $+$ when $\lambda<\tfrac12$ and $-$ when $\lambda>\tfrac12$.
\end{corollary}

The following observation identifies exactly when the spectral interval collapses to the right-simplex value.

\begin{lemma}\label{lem:interval}
For $1\le k\le n-1$ we have $\sigma_k^{+}\ge k\ge\sigma_k^{-}$, and equality holds in either
inequality if and only if $\mu_1=\cdots=\mu_n=1$, that is, if and only if the simplex is
right at $S$.
\end{lemma}

\begin{proof}
Put $\alpha=\sigma_k^{+}/k$ and $\beta=(n-\sigma_k^{+})/(n-k)$, the averages of the $k$
largest and of the $n-k$ smallest eigenvalues. Then
$\alpha\ge\mu_k\ge\mu_{k+1}\ge\beta$ and $k\alpha+(n-k)\beta=n$ by \eqref{eq:trM}, whence
$n\le n\alpha$ and $\alpha\ge1$, that is $\sigma_k^{+}\ge k$. If $\alpha=1$, then
$\beta=1$ as well, so $1=\alpha\ge\mu_k\ge\mu_{k+1}\ge\beta=1$; since $\alpha$ is the average
of $\mu_1,\ldots,\mu_k$, all of which are at least $\mu_k=1$, we get
$\mu_1=\cdots=\mu_k=1$, and symmetrically $\mu_{k+1}=\cdots=\mu_n=1$. The statement for
$\sigma_k^{-}$ follows from $\sigma_k^{-}=n-\sigma_{n-k}^{+}\le n-(n-k)=k$. Finally
$M=I$ is equivalent to $E$ being orthogonal, hence to the mutual orthogonality of the edges.
\end{proof}

\begin{remark}
Thus the interval $[\sigma_k^{-},\sigma_k^{+}]$ always contains $k$, and it collapses to the
single point $k$ exactly for right simplices. In this sense Theorem~\ref{thm:kyfan} is the
correct general statement: the identity of Theorem~\ref{thm:right} is the assertion that an
interval has length zero, and the converse results of the next section show that this happens only in the orthogonal case.
\end{remark}

\section{Converse results and quantitative estimates}\label{sec:converse}

Away from the midpoint value $\lambda=\tfrac12$, the right-simplex identity has a converse:
if it is independent of the projection subspace, then the edge directions must be orthogonal.
The same conclusion holds when one starts with arbitrary nonzero vectors rather than a
simplex.

\subsection{Characterisation of right simplices}

\begin{theorem}\label{thm:converse}
Let $\bu_1,\ldots,\bu_n$ be nonzero vectors in $\R^n$, let $1\le k\le n-1$ be fixed, and
let $\lambda\ne\tfrac12$ be fixed. The following are equivalent.
\begin{enumerate}
\item[\rm(i)] The vectors $\bu_1,\ldots,\bu_n$ are mutually perpendicular.
\item[\rm(ii)] For every $k$-dimensional subspace $L$ through $S$,
\[
\sum_{i=1}^{n}\frac{B_iH_i^2}{SA_i^2}=n\lambda^2-2k\lambda+k.
\]
\item[\rm(iii)] For every line $L$ through $S$,
\[
\sum_{i=1}^{n}\frac{SH_i^2}{SA_i^2}=1.
\]
\end{enumerate}
In particular, the vectors are then automatically linearly independent.
\end{theorem}

\begin{proof}
The implications (i)$\Rightarrow$(ii) and (i)$\Rightarrow$(iii) follow from
Theorem~\ref{thm:right} and Corollary~\ref{cor:lambda0}.

For (ii)$\Rightarrow$(i), Proposition~\ref{prop:master} and
$\lambda\ne\tfrac12$ show that $\tr(QM)=k$ for every rank $k$ orthogonal projection $Q$.
Let $\mu_1,\ldots,\mu_n$ be the eigenvalues of $M$. Choosing $L$ to be spanned by any
$k$ eigenvectors gives that every sum of $k$ eigenvalues equals $k$. Comparing two such
sums that differ in one index shows that all $\mu_i$ are equal. Since $\tr M=n$, we obtain
$M=I$, hence the unit vectors $\be_i$ form an orthonormal basis.

For (iii)$\Rightarrow$(i), write $L=\R\bw$ with $\norm\bw=1$. Then
$\tr(QM)=\bw^{\T}M\bw$, so (iii) says $\bw^{\T}(M-I)\bw=0$ for every unit $\bw$. Since a
symmetric matrix is determined by its quadratic form, $M=I$.
\end{proof}

\subsection{How many subspaces detect orthogonality?}

The previous theorem uses all $k$-dimensional subspaces. In fact, only finitely many are needed, and the next result gives the exact minimum.

\begin{theorem}\label{prop:howmany}
Fix $1\le k\le n-1$ and put
\[
D=\dim\operatorname{Sym}_n(\R)=\frac{n(n+1)}2.
\]
The least integer $N$ for which there exist $k$-dimensional subspaces
$L_1,\ldots,L_N$ with the following property is
\[
N=D-1=\frac{n(n+1)}2-1:
\]
whenever $\be_1,\ldots,\be_n$ are unit vectors and
\[
\sum_{i=1}^{n}\norm{Q_j\be_i}^2=k \text{ for } 1\le j\le N,
\]
where $Q_j$ denotes the orthogonal projection onto $L_j$, the vectors
$\be_1,\ldots,\be_n$ are an orthonormal basis.
\end{theorem}

\begin{proof}
Let $\operatorname{Sym}_0(n)$ denote the vector space of real symmetric $n\times n$ matrices with trace zero. Its dimension is $D-1$.

First suppose $N\le D-2$. The conditions $\tr(Q_jX)=0$ impose at most $N$ linear
conditions on $\operatorname{Sym}_0(n)$, so there is a nonzero matrix $X$ in $\operatorname{Sym}_0(n)$ that is orthogonal to all $Q_j$. For sufficiently small $t>0$,
$M_t=I+tX$ is positive definite, has trace $n$, and satisfies
$\tr(Q_jM_t)=k$ for every $j$. The eigenvalues of $M_t$ have sum $n$, so they majorise
$(1,\ldots,1)$. By the Schur--Horn theorem \cite{AHorn,HornJohnson}, there is an orthogonal
matrix $U$ such that $U^{\T}M_tU$ has diagonal entries all equal to $1$. The columns of
$M_t^{1/2}U$ are therefore unit vectors with frame operator $M_t$. Since $M_t\ne I$, they
are not pairwise orthogonal. Thus fewer than $D-1$ subspaces cannot suffice.

For the converse, consider the traceless matrices
\[
Q_L-\frac{k}{n}I,
\]
where $L$ ranges over all $k$-dimensional subspaces. They span $\operatorname{Sym}_0(n)$.
Indeed, suppose that $X$ is a trace-zero symmetric matrix orthogonal to all of them. Then
$\tr(Q_LX)=0$ for every $L$. Diagonalising $X$, with eigenvalues
$x_1,\ldots,x_n$, and taking $L$ to be spanned by any $k$ eigenvectors gives
For every $k$-element subset $J=\{j_1,\ldots,j_k\}$ we then have
\[
\sum_{r=1}^{k}x_{j_r}=0.
\]
Two such subsets differing in one index show that all $x_i$ are equal, and $\tr X=0$ then gives $X=0$. Hence one may choose $D-1$ subspaces
$L_1,\ldots,L_{D-1}$ such that the matrices
$Q_j-kI/n$ form a basis of $\operatorname{Sym}_0(n)$.

If the stated projection identities hold for these subspaces and $M$ is the frame operator
of the unit vectors, then $X=M-I$ is trace-zero and symmetric, and it is orthogonal to
that basis. Thus $X=0$ and $M=I$.
\end{proof}

For $k=1$, one may choose an explicit detecting family: the directions
$\varepsilon_1,\ldots,\varepsilon_{n-1}$ and
$(\varepsilon_i+\varepsilon_j)/\sqrt2$ for $1\le i<j\le n$, where
$\varepsilon_1,\ldots,\varepsilon_n$ is any fixed orthonormal basis.

\subsection{A sharp quantitative estimate}

For a simplex $SA_1\ldots A_n$, define
\[
\Theta:=\Bigl(\sum_{i\ne j}\cos^2\theta_{ij}\Bigr)^{1/2}
\]
and, for a line $L$ through $S$,
\[
\Delta(L)=\sum_{i=1}^{n}\frac{SH_i^2}{SA_i^2}-1.
\]

\begin{theorem}\label{thm:quantitative}
Put
\[
c_n=
\begin{cases}
1/\sqrt n,& n\text{ even},\\[2pt]
1/\sqrt{n-1},& n\text{ odd}.
\end{cases}
\]
Then
\begin{equation}\label{eq:quantitative}
c_n\Theta
\le
\max_L|\Delta(L)|
\le
\sqrt{\frac{n-1}{n}}\,\Theta,
\end{equation}
where the maximum is over all lines through $S$. Both constants are best possible.
\end{theorem}

\begin{proof}
Put $X=M-I$. For $L=\R\bw$ with $\norm\bw=1$,
\[
\Delta(L)=\bw^{\T}X\bw,
\]
so
\[
\max_L|\Delta(L)|=\norm{X}_{\mathrm{op}}.
\]
Moreover,
\[
\norm{X}_{\HS}^2
=
\tr(M^2)-n
=
\tr(G^2)-n
=
\sum_{i\ne j}\cos^2\theta_{ij}
=
\Theta^2.
\]
Thus it remains to compare the operator and Hilbert--Schmidt norms of a nonzero symmetric
matrix with trace zero.

Let $x_1,\ldots,x_n$ be the eigenvalues of $X$ and put
$m=\max_i|x_i|$. For the upper bound in \eqref{eq:quantitative}, assume after changing
sign that $x_1=m$. Since $\sum_{i=2}^{n}x_i=-m$, Cauchy--Schwarz gives
\[
\sum_{i=2}^{n}x_i^2\ge\frac{m^2}{n-1}.
\]
Hence
\[
\norm{X}_{\HS}^2\ge\frac{n}{n-1}m^2,
\]
which is the required upper estimate. Equality occurs for the spectrum
\[
\Bigl(m,-\frac{m}{n-1},\ldots,-\frac{m}{n-1}\Bigr).
\]

For the lower bound, scale so that $m=1$. If $n$ is even, then simply
$\sum_i x_i^2\le n$, giving $m\ge\norm{X}_{\HS}/\sqrt n$; equality occurs when half of the
eigenvalues are $1$ and half are $-1$. If $n=2r+1$ is odd, write the positive and negative eigenvalues separately. One of the two
sign classes contains at most $r$ terms. Since $|x_i|\le1$ and the positive and negative
sums have the same absolute value, that common sum is at most $r$. Hence
\[
\sum_i x_i^2\le\sum_i|x_i|\le2r=n-1.
\]
This gives $m\ge\norm{X}_{\HS}/\sqrt{n-1}$, with equality for the spectrum consisting of
$r$ copies of $1$, $r$ copies of $-1$, and one $0$.

These equality patterns are realised by simplices as well. For each of
the trace-zero spectra above, choose a symmetric matrix $Y$ with that spectrum. For
sufficiently small $\varepsilon>0$, $I+\varepsilon Y$ is positive definite with trace $n$.
The same Schur--Horn argument used in Theorem~\ref{prop:howmany} produces unit edge vectors
with frame operator $I+\varepsilon Y$. Since both ratios in \eqref{eq:quantitative} are
scale invariant in $Y$, the constants are sharp.
\end{proof}

\section{Affine projection subspaces}\label{sec:affine}

We now allow the projection subspace to miss the vertex $S$. Let
\[
L'=\mathbf p+L_0
\]
be a $k$-dimensional affine subspace, let $Q$ be the orthogonal projection onto its
direction space $L_0$, and let $F$ be the foot of the perpendicular from $S$ to $L'$. Put
\[
\begin{aligned}
\bq&=\overrightarrow{SF}=(I-Q)\mathbf p,\\
d&=\norm{\bq}=\dist(S,L').
\end{aligned}
\]
and define
\[
\bz=\sum_{i=1}^{n}\frac{\bu_i}{SA_i^2}.
\]
If $H_i$ is the orthogonal projection of $A_i$ onto $L'$, then
$\overrightarrow{SH_i}=Q\bu_i+\bq$.

\begin{lemma}\label{lem:affinegeneral}
For every simplex, every affine $k$-dimensional subspace $L'$, and every real $\lambda$,
\begin{equation}\label{eq:affinegeneral}
\sum_{i=1}^{n}\frac{B_iH_i^2}{SA_i^2}
=
n\lambda^2+(1-2\lambda)\tr(QM)
-2\lambda\langle\bz,\bq\rangle
+d^2\sum_{i=1}^{n}\frac1{SA_i^2}.
\end{equation}
\end{lemma}

\begin{proof}
Since $\bq\perp L_0$,
\[
\overrightarrow{B_iH_i}=(Q-\lambda I)\bu_i+\bq
\]
and
\[
B_iH_i^2
=
\norm{(Q-\lambda I)\bu_i}^2
-2\lambda\langle\bu_i,\bq\rangle+d^2.
\]
Divide by $SA_i^2$, sum, and use Proposition~\ref{prop:master}.
\end{proof}

For a right simplex the two correction terms have a direct geometric meaning. Let $h$ be
the altitude from $S$ to the hyperplane of the opposite facet $A_1\ldots A_n$.

\begin{lemma}\label{lem:altitude}
If the simplex is right at $S$, then $\langle\bz,\bu_j\rangle=1$ for every $j$, and
\[
\norm{\bz}^2=\sum_{i=1}^{n}\frac1{SA_i^2}=\frac1{h^2}.
\]
Thus $\bz$ is a normal vector to the opposite facet, normalised by
$\langle\bz,\mathbf x\rangle=1$ on that facet.
\end{lemma}

\begin{proof}
Orthogonality gives
\[
\langle\bz,\bu_j\rangle
=
\sum_i\frac{\langle\bu_i,\bu_j\rangle}{SA_i^2}=1.
\]
Hence the opposite facet lies in $\{\mathbf x:\langle\bz,\mathbf x\rangle=1\}$, whose
distance from the origin is $1/\norm\bz$. The displayed identity follows.
\end{proof}

The relation $h^{-2}=\sum_iSA_i^{-2}$ is the classical altitude formula for right
simplices; in dimension three it is the familiar companion to de Gua's theorem.

\begin{theorem}\label{thm:affine}
Let $SA_1\ldots A_n$ be right at $S$, let $L'$ be a $k$-dimensional affine subspace with
$1\le k\le n-1$, let $d=\dist(S,L')$, and let $\varphi$ be the angle between
$\overrightarrow{SF}$ and $\bz$ when $d>0$. Then, for every real $\lambda$,
\begin{equation}\label{eq:affine}
\sum_{i=1}^{n}\frac{B_iH_i^2}{SA_i^2}
=
n\lambda^2-2k\lambda+k
+\frac{d^2}{h^2}
-2\lambda\frac{d}{h}\cos\varphi.
\end{equation}
\end{theorem}

\begin{proof}
In Lemma~\ref{lem:affinegeneral} use $M=I$, Lemma~\ref{lem:altitude}, and
$\langle\bz,\bq\rangle=(d/h)\cos\varphi$.
\end{proof}

\begin{corollary}[$\lambda=0$]\label{cor:affinezero}
For every affine $k$-dimensional $L'$ at distance $d$ from $S$,
\[
\sum_{i=1}^{n}\frac{SH_i^2}{SA_i^2}=k+\frac{d^2}{h^2}.
\]
\end{corollary}

\begin{corollary}\label{cor:affinelower}
For every affine $k$-dimensional subspace $L'$ and every real $\lambda$,
\[
\sum_{i=1}^{n}\frac{B_iH_i^2}{SA_i^2}
\ge
(n-1)\lambda^2-2k\lambda+k.
\]
The constant is best possible. If $\lambda\ne0$, equality holds if and only if
\[
\begin{aligned}
\bz&\perp L_0,\\
\overrightarrow{SF}&=\lambda h^2\bz.
\end{aligned}
\]
If $\lambda=0$, equality holds if and only if $L'$ passes through $S$. In particular, for $\lambda=1$,
\[
\sum_{i=1}^{n}\frac{A_iH_i^2}{SA_i^2}\ge n-k-1.
\]
\end{corollary}

\begin{proof}
Put $t=d/h\ge0$. The correction in \eqref{eq:affine} satisfies
\[
t^2-2\lambda t\cos\varphi
\ge
t^2-2|\lambda|t
\ge
-\lambda^2.
\]
If $\lambda\ne0$, equality requires $t=|\lambda|$ and
$\cos\varphi=\operatorname{sign}(\lambda)$, which is equivalent to the stated vector
conditions. These conditions can be realised by choosing $L_0\subseteq\bz^\perp$. If
$\lambda=0$, equality is simply $t=0$, which means that $L'$ passes through $S$.
\end{proof}

\begin{remark}[Why the affine case needs correction terms]\label{rem:noaffine}
When $L'$ passes through $S$, the normalised quantity in \eqref{eq:identity} is unchanged by
a common dilation of the edge vectors. If $L'$ does not pass through $S$, the ratio $d/h$
changes under such a dilation. Therefore no identity with a constant right-hand side can
hold in the affine setting without recording the position of $L'$. Formula
\eqref{eq:affine} gives the exact correction. In particular, the midpoint value
$\lambda=\tfrac12$, degenerate for linear subspaces, becomes
\[
\sum_{i=1}^{n}\frac{M_iH_i^2}{SA_i^2}
=
\frac n4+\frac{d^2}{h^2}-\frac dh\cos\varphi.
\]
\end{remark}

\section{The vector form}\label{sec:vector}

Let $\bu_1,\ldots,\bu_n$ be nonzero mutually perpendicular vectors, $n\ge2$, and put
\[
\bv=\sum_{i=1}^{n}\bu_i.
\]
Since $\langle\bu_i,\bv\rangle=\norm{\bu_i}^2$, the orthogonal projection onto the line
$\R\bv$ is
\[
Q\bu_i=\frac{\norm{\bu_i}^2}{\norm\bv^2}\,\bv.
\]
Thus the original vector inequality is the $k=1$ case in which the projection line is the
diagonal determined by the vectors themselves.

\begin{theorem}\label{thm:vector}
For every real $t$,
\begin{equation}\label{eq:vec1}
\sum_{i=1}^{n}
\frac{\bigl\lVert\norm{\bu_i}^{2}\bv+t\norm\bv^{2}\bu_i\bigr\rVert^{2}}
{\norm{\bu_i}^{2}}
=
(nt^2+2t+1)\norm\bv^4,
\end{equation}
and consequently
\begin{equation}\label{eq:vec2}
\sum_{i=1}^{n}
\bigl\lVert\norm{\bu_i}^{2}\bv+t\norm\bv^{2}\bu_i\bigr\rVert
\le
\sqrt{nt^2+2t+1}\,\norm\bv^3.
\end{equation}
For $t=0$, equality in \eqref{eq:vec2} holds for every configuration. For $t\ne0$, equality
holds if and only if
\[
\norm{\bu_1}=\cdots=\norm{\bu_n}.
\]
Hence the constant is best possible for every $t$.
\end{theorem}

\begin{proof}
Theorem~\ref{thm:right} with $k=1$ and $\lambda=-t$ gives
\[
\sum_{i=1}^{n}\frac{\norm{Q\bu_i+t\bu_i}^2}{\norm{\bu_i}^2}
=nt^2+2t+1.
\]
Substituting the formula for $Q\bu_i$ and multiplying by $\norm\bv^4$ gives
\eqref{eq:vec1}. Equation \eqref{eq:vec2} follows from Cauchy--Schwarz and
$\sum_i\norm{\bu_i}^2=\norm\bv^2$.

For equality, put $a_i=\norm{\bu_i}$ and $S=\norm\bv^2$. Equality in the
Cauchy--Schwarz step is equivalent to
\[
\frac{\bigl\lVert a_i^2\bv+tS\bu_i\bigr\rVert}{a_i^2}
\]
being independent of $i$. A direct calculation gives
\[
\frac{\bigl\lVert a_i^2\bv+tS\bu_i\bigr\rVert^2}{a_i^4}
=
S\left(1+2t+\frac{t^2S}{a_i^2}\right).
\]
If $t=0$, this is independent of $i$ automatically. If $t\ne0$, it is independent of $i$
if and only if all $a_i$ are equal. The converse is immediate.
\end{proof}

\begin{corollary}\label{cor:vector}
With $\bv=\sum_i\bu_i$,
\[
\sum_{i=1}^{n}
\frac{\bigl\lVert\norm{\bu_i}^{2}\bv-\norm\bv^{2}\bu_i\bigr\rVert^{2}}
{\norm{\bu_i}^{2}}
=
(n-1)\norm\bv^4,
\]
and
\[
\sum_{i=1}^{n}
\bigl\lVert\norm{\bu_i}^{2}\bv-\norm\bv^{2}\bu_i\bigr\rVert
\le
\sqrt{n-1}\,\norm\bv^3.
\]
The constant $\sqrt{n-1}$ is best possible, with equality precisely when the orthogonal
vectors have equal lengths. Dividing by $\norm\bv^2$ recovers \eqref{eq:motivating} for the
particular line $SP$.
\end{corollary}

\section{Concluding remarks}\label{sec:conclusion}

The master identity provides a common framework for all the results above. In the orthogonal case, the trace term is
exactly the dimension of the projection subspace and yields a sharp family of distance
inequalities. For a general simplex, the Gram spectrum gives the sharp range of that term.
The converse results show that the collapse of this range is not accidental: it
characterises orthogonality, while the quantitative estimate measures how far the edge
directions are from being mutually perpendicular. The affine formula shows precisely
what changes when the projection subspace no longer passes through the distinguished
vertex.

Two directions seem natural for further study. First, the operator formulation suggests
looking for families of linear maps other than $Q-\lambda I$ whose Hilbert--Schmidt norm has
a comparably transparent geometric interpretation. Second, one may ask whether an analogue
of the affine identity exists for right simplices in spherical or hyperbolic geometry,
where orthogonal projection is no longer linear.

\section*{Acknowledgements}

The author thanks ChatGPT 5.6 Sol for assistance with the presentation of the manuscript and for checking some computations related to the constants appearing in the results. The mathematical content and final conclusions were verified by the author.

\begin{flushright}
	Quang Hung Tran,\\
	High School for Gifted Students,\\
	Vietnam National University, Hanoi, Vietnam,\\
	Email: tranquanghung@hus.edu.vn\\
	ORCID: 0000-0003-2468-4972
\end{flushright}


\begin{thebibliography}{99}

\bibitem{Berger}
M. Berger,
\textit{Geometry I},
Springer, Berlin, 1987.

\bibitem{Bhatia}
R. Bhatia,
\textit{Matrix Analysis},
Graduate Texts in Mathematics 169, Springer, New York, 1997.
doi:10.1007/978-1-4612-0653-8.

\bibitem{BjorckGolub}
{\AA}. Bj\"orck and G. H. Golub,
Numerical methods for computing angles between linear subspaces,
\textit{Math. Comp.} \textbf{27} (1973), 579--594.
doi:10.1090/S0025-5718-1973-0348991-3.

\bibitem{ConantBeyer}
D. R. Conant and W. A. Beyer,
Generalized Pythagorean theorem,
\textit{Amer. Math. Monthly} \textbf{81} (1974), 262--265.
doi:10.1080/00029890.1974.11993544.

\bibitem{Cook}
W. J. Cook,
An $n$-dimensional Pythagorean theorem,
\textit{College Math. J.} \textbf{44} (2013), no. 2, 98--101.
doi:10.4169/college.math.j.44.2.098.

\bibitem{Coxeter}
H. S. M. Coxeter,
\textit{Introduction to Geometry},
2nd ed., Wiley, New York, 1969.

\bibitem{DonchianCoxeter}
P. S. Donchian and H. S. M. Coxeter,
An $n$-dimensional extension of Pythagoras' theorem,
\textit{Math. Gaz.} \textbf{19} (1935), 206.
doi:10.2307/3605876.

\bibitem{Drucker}
D. Drucker,
A comprehensive Pythagorean theorem for all dimensions,
\textit{Amer. Math. Monthly} \textbf{122} (2015), no. 2, 164--168.
doi:10.4169/amer.math.monthly.122.02.164.

\bibitem{EiflerRhee}
L. Eifler and N. H. Rhee,
The $n$-dimensional Pythagorean theorem via the divergence theorem,
\textit{Amer. Math. Monthly} \textbf{115} (2008), 456--457.
doi:10.1080/00029890.2008.11920550.

\bibitem{Fan}
K. Fan,
On a theorem of Weyl concerning eigenvalues of linear transformations I,
\textit{Proc. Nat. Acad. Sci. U.S.A.} \textbf{35} (1949), 652--655.
doi:10.1073/pnas.35.11.652.

\bibitem{Fiedler}
M. Fiedler,
\textit{Matrices and Graphs in Geometry},
Cambridge University Press, Cambridge, 2011.
doi:10.1017/CBO9780511973611.

\bibitem{HLP}
G. H. Hardy, J. E. Littlewood and G. P\'olya,
\textit{Inequalities},
2nd ed., Cambridge University Press, Cambridge, 1952.

\bibitem{AHorn}
A. Horn,
Doubly stochastic matrices and the diagonal of a rotation matrix,
\textit{Amer. J. Math.} \textbf{76} (1954), 620--630.
doi:10.2307/2372705.

\bibitem{HornJohnson}
R. A. Horn and C. R. Johnson,
\textit{Matrix Analysis},
2nd ed., Cambridge University Press, Cambridge, 2013.

\bibitem{LinLin}
S.-Y. T. Lin and Y.-F. Lin,
The $n$-dimensional Pythagorean theorem,
\textit{Linear Multilinear Algebra} \textbf{26} (1990), 9--13.
doi:10.1080/03081089008817961.

\bibitem{Porter}
G. J. Porter,
$k$-volume in $\R^{n}$ and the generalized Pythagorean theorem,
\textit{Amer. Math. Monthly} \textbf{103} (1996), no. 3, 252--256.
doi:10.1080/00029890.1996.12004732.

\bibitem{TranKlamkin}
Q. H. Tran,
Some strengthened versions of Klamkin's inequality and applications,
\textit{Geom. Dedicata} \textbf{213} (2021), 467--472.
doi:10.1007/s10711-020-00591-x.

\bibitem{TranVanAubel}
Q. H. Tran,
Extending a theorem of van Aubel to the simplex,
\textit{J. Geom. Graph.} \textbf{25} (2021), no. 2, 253--263.

\bibitem{TranPythagoras}
Q. H. Tran,
A new proof of the $n$-dimensional Pythagorean theorem,
\textit{Math. Gaz.} \textbf{106} (2022), no. 565, 136--137.
doi:10.1017/mag.2022.27.

\bibitem{TranDeGua}
Q. H. Tran,
A generalization of de Gua's theorem with a vector proof,
\textit{Math. Intelligencer} \textbf{46} (2024), no. 3, 236--238.
doi:10.1007/s00283-023-10288-0.

\bibitem{TranGrace}
Q. H. Tran,
A simple proof of Grace's inequality in a right tetrahedron,
\textit{J. Geom.} \textbf{117} (2026), Article 17.
doi:10.1007/s00022-026-00800-0.

\end{thebibliography}
\end{document}